\documentclass[11pt, leqno]{amsart}
\usepackage{graphicx}
\usepackage{amsfonts,delarray,amssymb,amsmath,amsthm,a4,a4wide}
\usepackage{latexsym}
\usepackage{epsfig}
\usepackage{color}
\usepackage[margin=0.99in]{geometry} 
\newtheorem{thm}{Theorem}[section]

\newtheorem{lem}[thm]{Lemma}

\newtheorem{prop}[thm]{Proposition}
\theoremstyle{definition}

\newtheorem{rem}[thm]{Remark}
\numberwithin{equation}{section}
\newcommand{\norm}[1]{\left\Vert#1\right\Vert}

\newcommand{\R}{\mathbb R}

\newcommand{\ov}{\overline}
\newcommand{\p}{\partial}

\newcommand{\comment}[1]{}

\newenvironment{myindentpar}[1]%
{\begin{list}{}%
         {\setlength{\leftmargin}{#1}}%
         \item[]%
}
{\end{list}}

\begin{document}

\title[The prescribed affine mean curvature and Abreu's equations]{$W^{4, p}$ solution to the second boundary value problem of the prescribed affine mean curvature and Abreu's equations}
\author{Nam Q. Le}
\address{Department of Mathematics, Indiana University,
Bloomington, 831 E 3rd St, IN 47405, USA.}
\email{nqle@indiana.edu}

\thanks{The research of the author was supported in part by the National Science Foundation under grant DMS-1500400 and an Indiana University Summer Faculty Fellowship}
\subjclass[2010]{35J40, 35B65, 35J96, 53A15}
\keywords{Affine mean curvature equation, Abreu's equation, second boundary value problem, Legendre transform, linearized Monge-Amp\`ere equation}

\begin{abstract}
The second boundary value problem of the prescribed affine mean curvature equation is a nonlinear, fourth order, geometric partial differential equation. It was 
introduced by Trudinger and Wang in 2005 in their investigation of the affine Plateau problem in affine geometry. The previous works of Trudinger-Wang, Chau-Weinkove
and the author solved this global problem in $W^{4,p}$ under some restrictions on the sign or integrability of the affine mean curvature.
We remove these restrictions in this paper and obtain $W^{4,p}$ solution
to the second boundary value problem
when 
the affine mean curvature belongs to $L^p$ with $p$ greater than the dimension.  Our self-contained analysis also covers the case of Abreu's equation.

\end{abstract}
\maketitle

\section{Introduction and the main result}
In this paper, we are interested in obtaining global $W^{4,p}$ solution and 
$W^{4, p}$ estimates for the second boundary value problem of the prescribed affine mean curvature equation in dimensions $n\geq 2$. More generally, let
$G: (0,\infty) \rightarrow \mathbb{R}$ be a smooth, strictly increasing and strictly concave function on $(0, \infty)$. We consider
a fourth order, fully nonlinear, 
geometric partial differential equation of the form
\begin{equation}
\label{AMCE}
L[u]:= U^{ij}w_{ij} =f,\quad w = G'(\det D^{2} u)\quad \text{in}~\Omega,
\end{equation}
where  $\Omega\subset\R^n$ is a bounded, smooth and uniformly convex, $u$ is a locally uniformly convex function in $\overline{\Omega}$, and throughout, $$U = (U^{ij})=(\det D^{2} u) (D^{2} u)^{-1}$$ is the matrix of cofactors of the 
Hessian matrix $D^{2}u= (u_{ij})$.

We note that (\ref{AMCE}) consists of a Monge-Amp\`ere equation for $u$ in the form of $\det D^2 u=
G'^{-1}(w)$ and a linearized Monge-Amp\`ere equation for $w$ in the form of $U^{ij} w_{ij}=f$ because
the coefficient matrix $U$ comes from linearization of the Monge-Amp\`ere operator: $U=\frac{\p\det D^2 u}{\p u_{ij}}.$

The second boundary value problem for (\ref{AMCE}) prescribes the values of $u$ and its Hessian determinant $\det D^2 u$ on the boundary, or equivalently, 
\begin{equation}
\label{SBV}
u=\varphi,~~~ w=\psi~~~\text{on}~\p\Omega.
\end{equation}

The problem (\ref{AMCE})-(\ref{SBV}) with $G(d)=\frac{d^{\theta}-1}{\theta}$ and $\theta=\frac{1}{n +2}$ was introduced by Trudinger-Wang \cite{TW05} in their investigation of the affine Plateau problem in affine geometry. In this context, the quantity
$-\frac{1}{n+1} L[u]$
is the {\it affine mean curvature} of the graph of $u$; in particular, equation (\ref{AMCE}) with $f\equiv 0$ corresponds to the {\it affine maximal surface equation} \cite{TW00}. 
In the limiting case $\theta=0$ of $\frac{d^{\theta}-1}{\theta}$, we take $G(d)=\log d$ and (\ref{AMCE}) is then known as Abreu's equation in the 
context of existence of K\"ahler metric of constant scalar curvature \cite{Ab, CHLS, CLS, D1, D2, D3, D4, FS, Zhou, ZZ}. 

For a general concave function $G$, Donaldson \cite{D5} investigated local solutions of (\ref{AMCE}) with $f\equiv 0$ while Savin and the author \cite{LS}
studied regularity of (\ref{AMCE}) with Dirichlet and Neumann boundary conditions on $w$. In \cite{LS}, we considered (\ref{AMCE})
as an Euler-Lagrange equation of a Monge-Amp\`ere functional motivated by the Mabuchi functional in complex geometry. In fact, 
(\ref{AMCE}) is the Euler-Lagrange equation, with respect to compactly supported perturbations, of 
the functional
\begin{equation}\label{functional}
J[u]= \int_{\Omega}  G(\det D^2 u)dx -\int_{\Omega}  u fdx,
\end{equation}
defined over strictly convex functions $u$ on $\Omega$. For simplicity, we call $L[u]$ in (\ref{AMCE}), where $G$ is a general concave function, the {\it generalized affine mean curvature} 
of the graph of $u$.

It is an interesting problem, both geometrically and analytically, to study the solvability of the  fourth order, fully nonlinear equation
(\ref{AMCE})-(\ref{SBV}). As in the classical theory of second order elliptic equations, we are led naturally to the following:\\
{\bf Problem.} {\it Suppose the boundary data $\varphi$ and $\psi$ are smooth. Investigate the solvability of 
$C^{4,\alpha}(\overline{\Omega})$ solutions to (\ref{AMCE})-(\ref{SBV}) when $f$ is H\"older continuous and $W^{4,p}(\Omega)$ solutions when $f$ is less regular. }

Note that the case of dimension $n=1$ is very easy to deal with and is by now completely settled (see also \cite{CW}). Thus we assume throughout
that $n\geq 2$. Let us recall previous results on this problem in chronological order.

The situation of $C^{4,\alpha}(\overline{\Omega})$ solutions is by now well understood: Trudinger-Wang \cite{TW08} solved this problem 
when $f\in C^{\alpha}(\overline\Omega)$, $f\leq 0$, $G(d)=\frac{d^{\theta}-1}{\theta}$ and $\theta\in (0, 1/n)$. Very recently, Chau-Weinkove \cite{CW} completely 
removed the sign condition 
on $f$ in this case.

The situation of $W^{4,p}(\Omega)$ solutions is as follows:
\begin{myindentpar}{1cm}
$\bullet$ Trudinger-Wang \cite{TW08} solved this problem for all $p\in (1,\infty)$
when $f\in L^{\infty}(\Omega)$, $f\leq 0$, $G(d)=\frac{d^{\theta}-1}{\theta}$ and $\theta\in (0, 1/n)$.\\
$\bullet$ The author \cite{Le} solved this problem when $f\leq 0$, $f\in L^p(\Omega)$ with $p>n$, 
$G(d)=\frac{d^{\theta}-1}{\theta}$ and $\theta<1/n$.\\
$\bullet$ Chau-Weinkove \cite{CW} recently solved this problem for all $p\in (1,\infty)$ when $f\in L^{\infty}(\Omega)$, $G(d)=\frac{d^{\theta}-1}{\theta}$ and $0<\theta<1/n$, 
thus extending
the work of Trudinger-Wang \cite{TW08}. In this particular case of $G$ and $\theta$, they also solved this problem for $f\in L^ p(\Omega)$, $p>n$ with
$f^{+}:=\max(0, f)\in L^{q}(\Omega)$
for some $q>1/\theta$, thus extending the work of the author. Moreover, they also solved the problem for more general $G$ and $f$ with certain high integrability. 
More precisely, $f$ satisfies
 $\displaystyle{f \in L^{p}(\Omega)}$, $p>n$ and $f^+ \in L^{\infty}(\Omega)$ and $G: (0,\infty) \rightarrow \mathbb{R}$ together with its derivative $w(d)= G'(d)$
 satisfies in addition:
\begin{myindentpar}{2cm}
 (A1) $\displaystyle{w' + (1-\frac{1}{n}) \frac{w}{d} \le 0}$. \\
(A2) $d w \ge c>0$ for some $c>0$ and all $d\ge 1$.\\
(A3) $\displaystyle{d^{1-1/n} w \rightarrow \infty}$ as $d \rightarrow 0$.

\end{myindentpar}
\end{myindentpar}  

Chau-Weinkove raised the question \cite[Remark 1.1]{CW} on the weakest regularity assumption on $f$ giving a solution $u\in W^{4,p}$ to 
(\ref{AMCE})-(\ref{SBV}). 
\subsection{The main result}
Our main result, Theorem \ref{mainthm}, asserts the solvability of (\ref{AMCE})-(\ref{SBV}) in $W^{4,p}(\Omega)$ when $f\in L^p(\Omega)$ with $p>n$. In particular, it answers Chau-Weinkove's question 
under the weakest possible regularity on the generalized affine mean curvature $f$ and general concave function $G$ satisfying a set of conditions even weaker than (A1)-(A3).

From now on, we assume that $G: (0,\infty) \rightarrow \mathbb{R}$ is a  smooth strictly concave function on $(0,\infty)$ whose derivative $w(d)=G'(d)$ is strictly 
positive. We introduce the new coercivity condition: 
\begin{myindentpar}{3cm}
(B2) $G(d) - dG'(d)\rightarrow \infty$ when $d\rightarrow\infty$.
\end{myindentpar}

\begin{thm} \label{mainthm}
Fix $p>n$ and assume that (A1), (B2) and (A3) are satisfied.  Let $\Omega$ be a bounded, uniformly convex domain in $\R^n$ with $\partial \Omega \in C^{3, 1}$. 
Suppose $f \in L^{p}(\Omega)$, $\varphi \in W^{4,p}(\Omega)$ and $\psi \in W^{2,p}(\Omega)$ with $\inf_{\Omega}\psi>0$.    
Then there exists  a unique uniformly convex solution $u \in W^{4, p}(\Omega)$ to the second boundary value problem (\ref{AMCE})-(\ref{SBV}).
\end{thm}

It is quite remarkable that the integrability condition of the generalized affine mean curvature $L[u]$ in Theorem \ref{mainthm} does not depend on the concave function $G$. In the special case of
$G(d)=\frac{d^{\theta}-1}{\theta}$ with $\theta=\frac{1}{n+2}$, Theorem \ref{mainthm} tells us that we can prescribe the $L^p$ 
affine mean curvature (for any finite $p>n$) of the graph of a uniformly convex function with smooth Dirichlet boundary conditions on the function and its Hessian determinant.

By using the Leray-Schauder degree theory as in Trudinger-Wang \cite{TW05}, Theorem \ref{mainthm} follows from the following global a priori $W^{4,p}$ estimates for solutions 
of (\ref{AMCE})-(\ref{SBV}).
\begin{thm} 
\label{keythm} Assume that (A1), (B2) and (A3) are satisfied.
Let $p> n$ and let $\Omega$ be a bounded, uniformly convex 
domain in $\R^{n}$ with $\p\Omega\in C^{3,1}$. Suppose $\varphi \in W^{4,p}(\Omega), \psi\in W^{2,p}(\Omega)$, $\inf_{\Omega}\psi>0$ and $f\in L^{p}(\Omega)$. Then, 
for any uniformly convex solution $u\in C^{4}(\overline{\Omega})$ of (\ref{AMCE})-(\ref{SBV}), we have the estimate
\begin{equation}
\label{global-est}
\norm{u}_{W^{4,p}(\Omega)}\leq C,
\end{equation}
where $C$ depends 
on $n, p, G, \Omega$, $\norm{f}_{L^{p}(\Omega)}$, $\norm{\varphi}_{W^{4,p}(\Omega)}, \norm{\psi}_{W^{2,p}(\Omega)}$, and $\inf_{\Omega} \psi.$
\end{thm}

\begin{rem}A typical function $G$ satisfying (A1) and (A3) is $G(d) = \frac{d^{\theta}-1}{\theta}$ where $\theta<1/n$.
This example satisfies (A2) or (B2) only when $0\leq \theta<1/n$. 
\end{rem}

\begin{rem} \label{ABrem} The conditions (A1), (B2) and (A3) are weaker than (A1)-(A3). More precisely:
\begin{myindentpar}{1cm} 
 (a) Suppose that $G$ satisfies (A1)-(A3). Then,  (A1), (B2) and (A3) hold. In fact,
 $$G(d)- dG'(d)\geq G(1)- w(1) + c(1-\frac{1}{n}) \log d~\text{for all}~d\geq 1. $$
 (b) The following function satisfies (A1), (B2) and (A3) but does not satisfy (A1)-(A3):
 $$G(d) =\frac{\log d}{\log \log (d+ e^{e^{4n}})}.$$
 \end{myindentpar}
\end{rem}

Remark \ref{ABrem} (a) tells us that the functions $G$ satisfying (A1)-(A3) grow at least logarithmically, and thus belong roughly to the regime 
of $\frac{d^{\theta}-1}{\theta}$ where $0\leq \theta<1/n$ because 
(A1) also implies that $(wd^{1-1/n})'\leq 0$, which means that, after integrating twice, $G(d)\leq Cd^{1/n}$ for $d\geq 1$. 

On the other hand, the example in Remark \ref{ABrem} (b) says that Theorem \ref{mainthm} also covers the case of functions $G$ below 
the threshold of Abreu's equation where $G(d)=\log d$.

We briefly comment on the roles of conditions (A1), (B2) and (A3) in Theorem \ref{keythm}. (A1) guarantees the concavity of the 
functional $J$; (A3) gives the upper bound for $w$ while (B2) gives the upper bound for the dual $w^{*}$ of $w$ via the Legendre transform. As will be seen 
later in Lemma \ref{Legeq}, $w^{*}= G(\det D^2 u)- (\det D^2 u) G'(\det D^2 u).$ 
\subsection{Optimality of the assumptions}

We remark that the global $W^{4,p}$ estimate (\ref{global-est}) fails when
$f$ has integrability less than the dimension $n$ or the coercivity condition (B2) is not satisfied. 
\begin{rem} \label{smallpt}
Consider the second boundary value problem (\ref{AMCE})-(\ref{SBV}) on the unit ball centered at the origin $\Omega=B_1\subset \R^n$ with $G(d)=\frac{d^{\theta}-1}{\theta}$ where 
$\theta<1/n$.
\begin{myindentpar}{1cm}
(i) If $n/2< p<n$ then the function $u(x) =|x|^{1 + \frac{n-p}{2(3p-n)}}$
is a solution to (\ref{AMCE})-(\ref{SBV}) with $\theta = \frac{1}{2n}(\frac{n}{p}-1)$,
$\varphi$ and $\psi$ positive
 constants on $\p\Omega$, $f\in L^p(
 \Omega)$ 
 but
 $u\not\in W^{4, p}(\Omega)$.\\
 (ii) If $\theta<-1/n$, then $u(x) =|x|^{\frac{2n\theta}{n\theta-1}}$ is a solution to (\ref{AMCE})-(\ref{SBV}) with $\varphi$ and $\psi$ positive
 constants on $\p\Omega$ and $f$ a positive constant but
 $u\not\in W^{4, n/2}(\Omega)$.\\
 (iii) If $-1/n\leq \theta<0$, then $u(x) =|x|^{\frac{4-5n\theta}{4(1-n\theta)}}$ is a solution to 
 (\ref{AMCE})-(\ref{SBV}) with $\varphi$ and $\psi$ positive
 constants on $\p\Omega$ and $f\in L^{\frac{n}{1+n\theta/2}}(\Omega)$ but
 $u\not\in W^{4, n/2}(\Omega)$.
 
 \end{myindentpar}
\end{rem}
The situation in Remark \ref{smallpt} (ii) is similar to the 
case $\theta>1$ considered in \cite{TW02}. For
$n=2$, $\theta=2$ and $G(d) = \frac{d^{\theta}-1}{\theta}$, Trudinger-Wang \cite{TW02} 
constructed a radial function $u$, not $C^3$ smooth, with $U^{ij}w_{ij}$ being a positive constant. In the examples in Remark \ref{smallpt}, $w=G'(\det D^2u)$ vanishes somewhere in the interior of $\Omega$. It turns out that solutions 
to the second boundary value problem (\ref{AMCE})-(\ref{SBV}) 
are well-behaved near the boundary even if the coercivity condition (B2) fails. 
\begin{rem} 
\label{nearb} Let $G(d)=\frac{d^{\theta}-1}{\theta}$ with $\theta< 1/n.$
Let $p> n$ and let $\Omega$ be a bounded, uniformly convex 
domain in $\R^{n}$ with $\p\Omega\in C^{3,1}$. Suppose that $\varphi \in W^{4,p}(\Omega), \psi\in W^{2,p}(\Omega)$, $\inf_{\Omega}\psi>0$ and $f\in L^{p}(\Omega)$. Then, for any
uniformly convex solution $u\in C^{4}(\overline{\Omega})$ of (\ref{AMCE})-(\ref{SBV}), we have an a priori $W^{4,p}$ estimate near the boundary
\begin{equation*}
\norm{u}_{W^{4,p}(\Omega_{\delta})}\leq C,~\text{with}~\Omega_{\delta}:= \{x\in\Omega: \text{dist}(x,\p\Omega)
\leq \delta\},
\end{equation*}
where $C$ and $\delta$ depend on $n, p, \theta, \Omega$, $\norm{f}_{L^{p}(\Omega)}$, $\norm{\varphi}_{W^{4,p}(\Omega)}, \norm{\psi}_{W^{2,p}(\Omega)}$, and $\inf_{\Omega} \psi.$
\end{rem}

On the other hand, when $G(d) = \frac{d^{\theta}-1}{\theta}$ with $\theta<0$, that is the coercivity condition (B2) fails, but $\|f^{+}\|_{L^n(\Omega)}$ is small, then we still 
obtain a unique uniformly convex $W^{4, p}$ solution 
to the second boundary value problem (\ref{AMCE})-(\ref{SBV}). This is a 
slight improvement of \cite{TW05, TW08} where the smallness was taken in the $L^{\infty}$ norm.
\begin{rem} \label{Lnsmall}
Fix $p>n$ and let $G(d)=\frac{d^{\theta}-1}{\theta}$ with $\theta<1/n$.  Let $\Omega$ be a bounded, uniformly convex domain in $\mathbb{R}^n$ with $\partial \Omega \in C^{3, 1}$. 
Suppose $f \in L^{p}(\Omega)$, $\varphi \in W^{4,p}(\Omega)$, $\psi \in W^{2,p}(\Omega)$ with $\inf_{\Omega}\psi>0$.    
If $\|f^{+}\|_{L^n(\Omega)}$ is small, then there exists  a unique uniformly convex solution $u \in W^{4, p}(\Omega)$ to (\ref{AMCE})-(\ref{SBV}).
\end{rem}

\subsection{On the proof of global $W^{4,p}$ estimates}

In additions to global regularity results for the Monge-Amp\`ere \cite{TW08} and linearized Monge-Amp\`ere equations \cite{Le}, a key step to global $W^{4,p}$ estimates for (\ref{AMCE})-(\ref{SBV}) is to prove a uniform upper bound for the Hessian determinant $\det D^2 u$.

The condition $f\leq 0$ in \cite{TW08, Le} allowed the use of classical maximum principle to 
(\ref{AMCE}) to conclude that $w$ has a positive lower bound, and hence,  $\det D^2 u$ has a uniform upper bound.

When $f$ is not assumed to be non-positive, obtaining a uniform upper bound for $\det D^2 u$ becomes trickier. In \cite{CW}, with (A1)-(A3), by using Trudinger-Wang \cite{TW08} solution to (\ref{AMCE})-(\ref{SBV}) 
for $f\equiv 0$, Chau-Weinkove first obtained
a uniform global bound on $u$ by using very interesting geometric arguments involving the Gauss curvature of the boundary. From this, together with (A2), they obtained a uniform lower bound for
$w$ using either classical maximum principle or delicate Aleksandrov-Bakelman-Pucci (ABP) type argument to certain second order differential inequalities after performing some
(not sharp) pointwise
estimates. High integrability of $f^{+}$ was required in the process.

The key insight of this paper is that an application of ABP estimate to the  dual equation of (\ref{AMCE}) via the Legendre transform 
gives a uniform upper bound for $\det D^2 u$ provided we have the coercivity condition (B2) and a global gradient bound for $u$; see Lemma \ref{boundd}. We prove the latter by uniformly 
bounding $u$ globally and its Hessian determinant $\det D^2 u$ near the boundary. These estimates are derived as follows.

First, we give a direct proof of the global a priori bound on $u$ in Lemma \ref{udet}, assuming only (A1), without resorting to any previous results regarding to solvability
of (\ref{AMCE})-(\ref{SBV}). The proof is inspired by \cite[Lemma 2.2]{CW} but requires less regularity on the boundary data $\varphi$ and $\psi$ and gives explicit
bounds. It is
based on testing against smooth concave functions $\hat u$ and convex functions $\tilde u$ having generalized affine 
mean curvature $L[\tilde u]$  bounded in $L^1$ (Lemma \ref{geoc}). This proof is of independent interest and 
can be potentially applied to other problems concerning fourth order equations. By (A3), we 
have a uniform lower bound for 
$\det D^2 u$ (Lemma \ref{upwlem}). 
Next, by using our previous boundary H\"older estimates for second-order equations with lower bound on the determinant of the coefficient matrix \cite{Le} to $U^{ij}w_{ij}=f$, 
we obtain a uniform bound for $\det D^2 u$ near the boundary. This, together the global bound on $u$, 
allows us to construct barriers using the strict convexity of $\p\Omega$ to obtain the global gradient bound for $u$;
see Lemma \ref{Dubound}. 

\section{Proof of the main theorem}

In this section, we give the proof of the main result of the paper, Theorem \ref{mainthm}. As mentioned in the Introduction, to prove Theorem \ref{mainthm}, it suffices to prove Theorem \ref{keythm}.

Let $p> n$ and let $\Omega$ be a bounded, uniformly convex 
domain in $\R^{n}$ with $\p\Omega\in C^{3,1}$. Assume $\varphi \in W^{4,p}(\Omega), \psi\in W^{2,p}(\Omega)$, $\inf_{\Omega}\psi>0$ and $f\in L^{p}(\Omega)$. 
Suppose a uniformly convex solution $u\in C^{4}(\overline{\Omega})$ solves (\ref{AMCE})-(\ref{SBV}).

We denote by $C, C', C_1, C_2, c, c_1$, etc,  {\it universal constants} that may change from line to line. Unless stated otherwise, they depend only 
on $n, p, G, \Omega$, $\norm{f}_{L^{p}(\Omega)}$, $\norm{\varphi}_{W^{4,p}(\Omega)}, \norm{\psi}_{W^{2,p}(\Omega)}$, and $\inf_{\Omega} \psi.$

Our basic geometric construction is the following:
\begin{lem}\label{geoc} There exist a convex function $\tilde u\in W^{4,p}(\Omega)$ and a concave function $\hat u\in W^{4,p}(\Omega)$
 with the following properties:
 \begin{myindentpar}{1cm}
  (i) $\tilde u=\hat u=\varphi$ on $\p\Omega$,\\
  (ii) $$
\| \tilde{u} \|_{C^{3}(\ov{\Omega})} + \| \hat{u} \|_{C^{3}(\ov{\Omega})} +
\|\tilde u\|_{W^{4, p}(\Omega)} + \|\hat u\|_{W^{4, p}(\Omega)}\leq C,\quad \textrm{and } \det D^2\tilde{u} \ge C^{-1}>0,
$$
(iii) letting $\tilde{w}=G'(\det D^2 \tilde{u})$, and denoting by $(\tilde{U}^{ij})$ the 
cofactor matrix of $(\tilde{u}_{ij})$, then  the generalized affine mean curvature of the graph of $\tilde u$ is uniformly bounded in $L^p$, that is $$\left\|\tilde U^{ij}\tilde w_{ij}\right\|_{L^p(\Omega)}\leq C,$$
 \end{myindentpar}
where $C$ depends only on $n$, $p$, $\Omega$, 
$G$, and $\| \varphi\|_{W^{4, p}(\Omega)}$.
 
\end{lem}

\begin{proof} 
Let $\rho$ be a strictly convex defining function of $\Omega$, that is 
$\Omega:=\{x\in \R^n: \rho(x)<0\}$, $\rho=0$ on $\p\Omega$ and $D\rho\neq 0$ on $\p\Omega$.
Then $$D^2 \rho\geq \eta I_n~ \text{and} ~\rho\geq -\eta^{-1} ~\text{in}~ \Omega$$ for some $\eta>0$ depending only on $\Omega$. 
Consider the following functions
$$\tilde u(x) =\varphi(x) + \mu (e^{\rho}-1), \hat{u}= \varphi(x) - \mu (e^{\rho}-1).$$
Then $\tilde u, \hat u \in W^{4,p}(\Omega)$. 
From
$$D^2 (e^{\rho}-1)=e^{\rho}(D^2\rho + D\rho \otimes D\rho)\geq e^{-\eta^{-1}}\eta I_n,$$
we find that for a fixed but sufficiently large $\mu$ (depending only on $n, p, \Omega$ and $\|\varphi\|_{W^{4, p}(\Omega)}$), $\tilde u$ is convex while $\hat u$ is concave and; moreover, recalling $p>n$, (i) and (ii) are satisfied. 
From (ii), the 
smoothness of $G$ and 
$$\tilde w_{ij}= G'''(\det D^2\tilde u) \tilde U^{kl}\tilde U^{rs} \tilde u_{kli}\tilde u_{rsj} + G''(\det D^2\tilde u) \tilde U^{kl}\tilde u_{klij}
+ G''(\det D^2\tilde u) \tilde U^{kl}_j \tilde u_{kli},$$
we easily obtain (iii).
\end{proof}

The following lemma gives a uniform $L^1$ bound on the Hessian determinant $\det D^2 u$ and as a consequence, a uniform bound on $u$. 
\begin{lem}\label{udet} Assuming (A1), we have 

$$(i)~~\int_{\Omega} \det D^2 u \le C,~\text{and}~
(ii)~ \sup_{\Omega} |u| \le C,$$
where $C$ depends only on $n$, $p$, $\Omega$, 
$G$, $\| f\|_{L^1(\Omega)}$, $\| \varphi\|_{W^{4, p}(\Omega)}$, $\| \psi\|_{L^{\infty}(\ov{\Omega})}$ and $\inf_{\partial \Omega}\psi$.

\end{lem}
\begin{rem}
Lemma \ref{udet} (ii) was proved in \cite[Lemma 2.2]{CW} under (A1)-(A3) using the result of Trudinger-Wang \cite{TW08} on the solvability of (\ref{AMCE})-(\ref{SBV})
in $W^{4,p}(\Omega)$ when $f\equiv 0$. The constant
$C$ in \cite{CW} depends on $n$, $p$, $G$, $\Omega$, $\| f\|_{L^1(\Omega)}$, $\| \varphi\|_{C^{3, 1}(\overline{\Omega})}$, 
$\| \psi\|_{C^{1, 1}(\ov{\Omega})}$ and $\inf_{\partial \Omega}\psi$.
Our proof here is self-contained and in fact gives a stronger bound (i). Moreover, from the proofs of Lemmas \ref{geoc} and \ref{udet}, we find that the constant $C$ can be made explicit.
\end{rem}
\begin{proof}[Proof of Lemma \ref{udet}]  
Let $\tilde u$ be as in Lemma \ref{geoc}. Set $\tilde f= \tilde U^{ij}\tilde w_{ij}$.
The assumption (A1) implies that the function $\tilde G(d):= G(d^n)$ is concave because
$$\tilde G^{''}(d)= n^2 d^{n-2} \left[w'(d^n) d^n + (1-\frac{1}{n})w(d^n)\right]\leq 0.$$
Using this, $G'>0$, and the concavity of the map $M\longmapsto (\det M)^{1/n}$ in the space of symmetric matrices $M\geq 0$, we obtain
\begin{eqnarray*}
 \tilde G((\det D^2 \tilde u)^{1/n}) -\tilde G((\det D^2  u)^{1/n})&\leq& \tilde G^{'}((\det D^2 u)^{1/n})((\det D^2 \tilde u)^{1/n}-(\det D^2 u)^{1/n})\\
 &\leq & \tilde G^{'}((\det D^2 u)^{1/n})\frac{1}{n} (\det D^2 u)^{1/n-1} U^{ij} (\tilde u-u)_{ij} .
\end{eqnarray*}
Since $\tilde G^{'}((\det D^2 u)^{1/n}) = n G^{'}(\det D^2 u) (\det D^2 u)^{\frac{n-1}{n}}$, we rewrite the above inequalities as
$$G(\det D^2 \tilde u)- G(\det D^2 u)\leq wU^{ij}(\tilde u- u)_{ij}.$$
Similarly,
$$G(\det D^2 u)- G(\det D^2 \tilde u)\leq \tilde w \tilde U^{ij}(u- \tilde u)_{ij}.$$
Adding, integrating by parts twice and using the fact that $(U^{ij})$ is divergence free, we obtain
\begin{eqnarray*}
 0&\leq& \int_{\Omega} wU^{ij}(\tilde u- u)_{ij} + \tilde w \tilde U^{ij}(u- \tilde u)_{ij}\\
 &=& \int_{\partial \Omega} \psi U^{ij} (\tilde{u}_j - u_j) \nu_i + \int_{\Omega} f(\tilde{u}-u) + 
 \int_{\partial \Omega} \tilde w \tilde{U}^{ij} (u_j - \tilde{u}_j) \nu_i + \int_{\Omega} \tilde f (u-\tilde{u}) . 
\end{eqnarray*}
Here $\nu= (\nu_{1},\cdots,\nu_n)$ is the unit outer normal vector field on $\p\Omega$. It follows that
\begin{eqnarray} \label{key1}
\int_{\Omega} (f- \tilde f)u + \int_{\partial \Omega} \left(\psi U^{ij} (u_j -\tilde{u}_j ) \nu_i  +
 \tilde w \tilde{U}^{ij} (\tilde{u}_j - u_j) \nu_i\right) &\le& \int_{\Omega} (f-\tilde f) \tilde u\nonumber\\&\leq& 
(\|f\|_{L^1(\Omega)} + \|\tilde f\|_{L^1(\Omega)})\|\tilde u\|_{L^{\infty}(\Omega)}\leq 
C.
\end{eqnarray}
Let us analyze the boundary terms in (\ref{key1}). Since $u-\tilde u=0$ on $\p\Omega$, we have $(u-\tilde u)_j= (u-\tilde u)_{\nu} \nu_j$, and hence
$$U^{ij}(u-\tilde u)_j \nu_i= U^{ij}\nu_j \nu_i (u-\tilde u)_{\nu} = U^{\nu\nu}(u-\tilde u)_{\nu}\equiv (\det D^2_{x'} u)(u-\tilde u)_{\nu},$$
with $x'\perp \nu$ denoting the tangential directions along $\p\Omega$. Therefore, 

\begin{equation} \label{nunu}
U^{ij} (u_j - \tilde{u}_j) \nu_i = U^{\nu\nu} (u_{\nu} - \tilde{u}_{\nu}), \quad \tilde{U}^{ij} (\tilde{u}_j - u_j) \nu_i = \tilde{U}^{\nu\nu} (\tilde{u}_{\nu} - u_{\nu}).
\end{equation}

On the other hand, from $u-\varphi=0$ on $\p\Omega$, we have, with respect to a principle coordinate system at any point $y\in\p\Omega$ (see, e.g., 
\cite[formula (14.95) in \S 14.6]{GT})
$$D_{ij}(u-\varphi)= (u-\varphi)_{\nu}\kappa_i\delta_{ij}, i, j=1, \cdots, n-1,$$
where $\kappa_1,\cdots,\kappa_{n-1}$ denote the principle curvatures of $\p\Omega$ at $y$. Let $K=\kappa_1\cdots\kappa_{n-1}$ be the Gauss curvature of $\partial \Omega$ at
$y\in\p\Omega$. Then, at any $y\in\Omega$, by noting that 
$\det D^2_{x'} u =\det (D_{ij}u)_{1\leq i, j\leq n-1}$ and taking the determinants of
$$D_{ij} u = u_{\nu}\kappa_i\delta_{ij}-\varphi_{\nu}\kappa_i\delta_{ij} + D_{ij}\varphi,$$
we obtain, with $u_{\nu}^+ =\max (0, u_{\nu})$,
\begin{equation} \label{Gaussc}
U^{\nu\nu}  = K (u_{\nu})^{n-1} + E, \quad \textrm{where } |E| \le C (1+ |u_{\nu}|^{n-2}) \leq C(1 + (u_{\nu}^+)^{n-2}).
\end{equation}
In the last inequality of (\ref{Gaussc}), we used the following fact which is due to the convexity of $u$:
\begin{equation}
 \label{lowerunu} u_{\nu}\geq -\|D\varphi\|_{L^{\infty}(\Omega)}.
\end{equation}

Now, let $\hat u$ be as in Lemma \ref{geoc}. Integrating by parts twice, and using (\ref{nunu}), we find that
\begin{equation}\label{hateq}\int_{\Omega} U^{ij} (u-\hat u)_{ij}= \int_{\p\Omega} U^{ij} (u-\hat u)_{i}\nu_j= \int_{\p\Omega} U^{\nu\nu} (u_{\nu}-\hat u_{\nu}).
\end{equation}
By Lemma \ref{geoc}, $\hat u_{\nu}$ is bounded by a universal constant. 
The concavity of $\hat u$ gives $U^{ij} \hat u_{ij}\leq 0$. Thus,
using $U^{ij} u_{ij}= n\det D^2 u$, we obtain from (\ref{Gaussc})-(\ref{hateq}) the following estimates
\begin{equation}\int_{\Omega} \det D^2 u \leq  \int_{\p\Omega} U^{\nu\nu} (u_{\nu}-\hat u_{\nu}) \leq
C + C
\int_{\p\Omega}  (u_{\nu}^+)^n .
 \label{l1det}
\end{equation}
The Aleksandrov's maximum principle (see \cite[Lemma 9.2]{GT}) then gives

\begin{equation}\label{umax}\|u\|_{L^{\infty}(\Omega)} \leq \|\varphi\|_{L^{\infty}(\p\Omega)} + C(n) \text{diam}(\Omega)\left(\int_{\Omega} \det D^2 u\right)^{1/n} \leq C + C\left(
\int_{\p\Omega}  (u_{\nu}^+)^n \right)^{1/n}.
\end{equation}

By Lemma \ref{geoc}, $\tilde{u}_{\nu}, \tilde w$, $\tilde{U}^{\nu\nu}$ and $\|\tilde f\|_{L^1(\Omega)}$ are uniformly bounded. Taking (\ref{key1})-(\ref{Gaussc}) 
and (\ref{umax}) into account, we obtain
\begin{eqnarray*} 
 \int_{\partial \Omega} K \psi (u_{\nu}^+)^n  &\le& C + C \int_{\partial \Omega} (u_{\nu}^+)^{n-1} - \int_{\Omega} (f - \tilde f)u\\
 &\leq& C + C \int_{\partial \Omega} (u_{\nu}^+)^{n-1} + C\left(
\int_{\p\Omega}  (u_{\nu}^+)^n \right)^{1/n}.
\end{eqnarray*} 
From H\"older inequality, $n\geq 2$ and the fact that $K\psi$ has a positive lower bound, we easily obtain
$$\int_{\partial \Omega} (u_{\nu}^+)^n \le C,$$
from which the claimed uniform bound for $u$ in (i) follows by (\ref{umax}).
Recalling (\ref{l1det}), we obtain 
the desired bound for the $L^1$ norm of $\det D^2 u$ stated in (i).
\end{proof}
The next lemma gives a uniform lower bound on the Hessian determinant $\det D^2 u$. 
\begin{lem} (\cite[Lemma 2.3]{CW}, \cite[Lemma 3.1]{Le}) \label{upwlem} Assume (A3) is satisfied. 
Then, there exists a constant $C>0$ depending 
only on $n, p, G, \Omega$, $\norm{f}_{L^{n}(\Omega)}$ and $\norm{\psi}_{W^{2,p}(\Omega)}$ 
such that 
\begin{equation*}
w\leq C,~\text{and}~ \det D^2 u\geq C^{-1}.
\end{equation*}
\end{lem}
The proof of Lemma \ref{upwlem} is short, so we include it here for reader's convenience.
\begin{proof} Let $d:=\det D^2 u$. Then $\det U = d^{n-1}.$ 
We apply Aleksandrov-Bakelman-Pucci's maximum principle (see e.g. \cite[Theorem 9.1]{GT}) to $U^{ij}w_{ij}=f$ in $\Omega$ with
$w=\psi$ on $\p\Omega$ to find that
\begin{equation}\label{Aleksandrov}\sup_{\Omega} w \leq \sup_{\partial \Omega} \psi +C \left\|\frac{f}{d^{(n-1)/n}}\right\|_{L^n(\Omega)}
\leq \sup_{\partial \Omega} \psi +C \left\| f\right \|_{L^n(\Omega)} \sup_{\Omega} (d^{(1-n)/n}),\\
\end{equation}
where $C$ depends only on $n$ and $\Omega$.  The desired upper bound on $w$ follows from \eqref{Aleksandrov} and  assumption (A3) on $G$.  The lower bound for $\det D^2 u=d$ then follows immediately.
\end{proof}
Now, we prove a key gradient bound for $u$.
\begin{lem}\label{Dubound}
 Assume (A1), and (A3) are satisfied. Then, there exists a constant $C>0$ depending 
only on $n, p, G, \Omega$, $\norm{f}_{L^{n}(\Omega)}$, $\|\varphi\|_{W^{4, p}(\Omega)}$, $\norm{\psi}_{W^{2,p}(\Omega)}$ and $\inf_{\p\Omega}\psi$
such that 
 $$\sup_{\Omega}|Du|\leq C.$$ 
\end{lem}
To this end, we recall the following result on boundary H\"older estimates for second-order elliptic equations with lower bound on the determinant of the coefficient matrix.
\begin{prop} (\cite[Proposition 2.1]{Le})\label{Hlem}
Let $v$ be the continuous solution to the equation $a_{ij} v_{ij} = g$ in $\Omega$ with $v=\varphi$ on $\partial\Omega$. Here the matrix $(a_{ij})$ is assumed to be measurable, positive definite and satisfies $\det (a_{ij})\geq \lambda.$ We assume that
$$\|\varphi\|_{L^{\infty}(\Omega)} + \|g\|_{L^{n}(\Omega)}\leq K$$
and $\varphi$ is $C^{\alpha}$ on $\partial\Omega$, namely,
$$|\varphi(x)-\varphi(y)|\leq L |x-y|^{\alpha}~\text{for}~x, y\in\partial\Omega.$$ Then, there exist $\delta, C$ depending only on $\lambda, n, K, L, \alpha$, $diam (\Omega)$, and the uniform convexity of $\Omega$ so that, for any $x_{0}\in\partial\Omega$, we have
$$|v(x)-v(x_{0})|\leq C|x-x_{0}|^{\frac{\alpha}{\alpha +2}},~\forall x\in \Omega\cap B_{\delta}(x_{0}). $$
\end{prop}
\begin{proof} [Proof of Lemma \ref{Dubound}]
Let $\nu$ be the unit outer normal vector field on $\p\Omega$.
The crucial point in the proof is to prove an upper bound for $u_{\nu}$. 

By Lemma \ref{upwlem}, we have
a lower bound for the Hessian determinant $\det D^2 u\geq C_1$ 
where $C_1$ depends
only on $n, p, G, \Omega$, $\norm{f}_{L^{n}(\Omega)}$ and $\norm{\psi}_{W^{2,p}(\Omega)}$. Because $p>n$, $\psi$ is clearly H\"older continuous in $\overline{\Omega}$. Since $\det U\geq C_1^{n-1}$, applying  Proposition \ref{Hlem} to $U^{ij}w_{ij}= f$ in $\Omega$, we find that
$w$ is H\"older continuous at the boundary. 

Note that (A1) implies  $(w(d)d^{1-1/n})' \le 0$ and therefore
$w(d)d^{1-1/n} \le C$ for $ d\geq 1.$
Since $w=\psi\geq \inf_{\p\Omega} \psi>0$ on 
$\p\Omega$, it follows from the boundary H\"older continuity of $w$ that $w$ is uniformly bounded from below while $\det D^2 u$  is uniformly bounded from above in a 
neighborhood $\Omega_{\delta}:= \{x\in\Omega: \text{dist}(x,\p\Omega)
\leq \delta\}$ of the boundary. Here $\delta$ is a universal constant, depending only on $n, p, G,\Omega$, $\inf_{\p\Omega} \psi$, $\|f\|_{L^n(\Omega)}$ and $\norm{\psi}_{W^{2,p}(\Omega)}$.

Let $\rho$ be a strictly convex defining function of $\Omega$, that is 
$\Omega:=\{x\in \R^n: \rho(x)<0\}$, $\rho=0$ on $\p\Omega$ and $D\rho\neq 0$ on $\p\Omega$. Then
$$D^2\rho\geq \eta I_n~\text{and}~\rho\geq -\eta^{-1}~\text{in}~\Omega$$
where $\eta>0$ depends only on $\Omega$. We easily find that,
for large $\mu$, the function
$$v(x) =\varphi(x) + \mu (e^{\rho}-1)$$
is a lower bound for $u$ in $\Omega_{\delta}$.

Indeed, there 
exists a universal $C_2>0$ such that $\rho\leq -C_2$ on $\p\Omega_{\delta}
\cap \Omega$.
Since $$ D^2 v = D^2\varphi + \mu e^{\rho}(D^2\rho + D\rho \otimes D\rho)\geq D^2\varphi + \mu\eta e^{-\eta^{-1}} I_n~\text{in}~ \Omega_{\delta},$$ $v=u$ on $\p\Omega$ while
$$v\leq \|\varphi\|_{L^{\infty}(\Omega)} + \mu (e^{-C_2}-1)~ \text{on}~ \p\Omega_{\delta}
\cap \Omega,$$ we find that for $\mu$ universally large,
$$\det D^2 v\geq \det D^2 u~\text{in} ~\Omega_{\delta}$$
and $u\geq v$ on $\p\Omega_{\delta}$ by the global bound on $u$ in Lemma \ref{udet}. Hence $u\geq v$ in $\Omega_{\delta}$ by the comparison principle (see 
\cite[Theorem 17.1]{GT}). 

From $u=v$ on $\p\Omega$, we deduce that $u_{\nu}\leq v_{\nu}$ and this gives a uniform upper bound for $u_{\nu}$. By convexity, 
$$u_{\nu}(x) \ge  -\|D\varphi\|_{L^{\infty}(\Omega)}, \quad \textrm{for all } x\in \partial \Omega. $$ Because $u=\varphi$ on $\p\Omega$, the tangential derivatives of $u$ on $\p\Omega$ are those of $\varphi$. 
Thus $Du$ is uniformly bounded on $\p\Omega$. Again, by convexity, we find that $Du$ is bounded in $\Omega$ by a universal constant as stated in the lemma.
\end{proof}

Finally, we come to the key argument of the paper. To prove a uniform upper bound for $\det D^2 u$, we use the Legendre transform:
$$y= Du(x), u^*(y) = x\cdot y - u(x) (=\sup_{z\in\Omega} \left(y\cdot z-u(z))\right).$$
The Legendre transform $u^{*}$ of $u$ is defined in $\Omega^*:=Du(\Omega)$. $u^*$ is a uniformly convex, $C^4$ smooth function in
$\Omega^*$. Furthermore the Legendre transform of $u^*$ is $u$
itself. From $y=Du(x)$ we have $x=D u^*(y)$ and
$D^2 u(x)= \left(D^2 u^*(y)\right)^{-1}. $

The Legendre transform $u^{*}$ satisfies a dual equation to (\ref{AMCE}) as stated in the following lemma.
\begin{lem}\label{Legeq}
The Legendre transform $u^*$ satisfies the equation
$${U^*}^{ij} w^*_{ij}=-f(Du^*)\,\det D^2 u^*,$$
where $({U^*}^{ij})$ is the cofactor matrix of $D^2 u^*$ and $$w^*=G\left((\det
D^2 u^*)^{-1}\right)- (\det
D^2 u^*)^{-1} G' \left((\det
D^2 u^*)^{-1}\right).$$
\end{lem}

This lemma was previously observed by Trudinger-Wang \cite{TW00} (in the proof of Lemma 3.2 there) and Zhou \cite{Zhou} (before the proof of Lemma 3.2 there). The idea is as follows.
Since $u$ is a critical point of the functional $J$, $u^*$ is a critical point of the dual functional $J^*$ under local perturbations,
so it satisfies the Euler-Lagrange equation
of $J^*$ and this gives the conclusion of Lemma \ref{Legeq}. 
 The dual functional with respect to the Legendre transform is given by
$$J^*(u^*)= \int_{\Omega^*}G([\det D^2u^*]^{-1})\det D^2u^*\, dy
-\int_{\Omega^*}f(Du^*)(yDu^*-u^*)\det D^2u^*\, dy.$$

We give here a direct proof of Lemma \ref{Legeq}.

\begin{proof}[Proof of Lemma \ref{Legeq}]
For simplicity, let
$d= \det D^2 u$ and $d^*= \det D^2 u^*.$ Then $d= {d^{*}}^{-1}$.
We denote by $(u^{ij})$ and $({u^*}^{ij})$ the inverses of the Hessian matrices $D^2 u= (u_{ij})= (\frac{\p^2 u}{\p x_i\p x_j})$ and $D^2 u^*=(u^*_{ij})=(\frac{\p^2 u^*}{\p y_i\p y_j})$.
Note that $w= G' (d)= G' ({d^{*}}^{-1}).$ Thus
$$w_{j} =\frac{\p w}{\p x_j} =\frac{\p G' ({d^*}^{-1})}{\p y_k} \frac{\p y_k}{\p x_j} = \left[\frac{\p }{\p y_k} G' ({d^*}^{-1})\right]u_{kj}=
\left[\frac{\p }{\p y_k} G' ({d^*}^{-1})\right]{u^*}^{kj}.$$
Clearly,

$${d^{*}}^{-1}\frac{\p }{\p y_k} G' ({d^*}^{-1}) = -\frac{\p}{\p y_k} \left[G ({d^*}^{-1})- {d^*}^{-1}G' ({d^*}^{-1})\right ]= - w^{*}_k,$$
from which it follows that
$w_j = - w^{*}_k (U^{*})^{kj}.$
Similarly,
$w_{ij}= \frac{\p }{\p y_l} w_j {u^*}^{li}.$
Hence, using
$$U^{ij} = (\det D^2 u) u^{ij}= (d^{*})^{-1} u^{*}_{ij},$$
and the fact that $U^{*}= ({U^*}^{ij})$ is divergence-free, we find from (\ref{AMCE}) that
\begin{eqnarray*}f(Du^{*})= U^{ij} w_{ij}=(d^{*})^{-1} u^{*}_{ij} {u^*}^{li} \frac{\p }{\p y_l} w_j  
= -(d^{*})^{-1}  \frac{\p }{\p y_j}\left\{ w^{*}_k
{U^*}^{kj}\right\} = -(d^{*})^{-1} {U^{*}}^{kj} w^{*}_{kj}.
\end{eqnarray*}
Thus, the lemma is proved.
\end{proof}

We are now ready to prove that the Hessian determinant $\det D^2 u$ is universally bounded away from $0$ and $\infty$.
\begin{lem} \label{boundd} Assume (A1), (B2) and (A3) are satisfied. Then, there exists a constant $C>0$ depending 
only on $n, p, G, \Omega$, $\norm{f}_{L^{n}(\Omega)}$, $\|\varphi\|_{W^{4, p}(\Omega)}$, $\norm{\psi}_{W^{2,p}(\Omega)}$ and $\inf_{\p\Omega}\psi$
such that 
 $$C^{-1}\leq \det D^2 u\leq C.$$
\end{lem}
\begin{proof}[Proof of Lemma \ref{boundd}] We use the same notation as in Lemma \ref{Legeq} and its proof. 
By Lemma \ref{Dubound}, $diam(\Omega^*)$ is bounded by a universal constant $C$. With (A1) and (A3), we can apply Lemma \ref{Legeq} to conclude that
${u^{*}}^{ij} w^{*}_{ij} = -  f(Du^{*}(y))~\text{in}~\Omega^*$
with $$w^*=G(d) - d G'(d) = G({G^{'}}^{-1}(w))- {G^{'}}^{-1}(w) w= G({G^{'}}^{-1}(\psi))- {G^{'}}^{-1}(\psi) \psi~ \text{on}~ \p\Omega^*.$$
Applying the ABP estimate \cite[Theorem 9.1]{GT} to $w^{*}$ on $\Omega^{*}$, and then changing of variables $y= Du(x)$ with $dy = \det D^2 u~ dx,$ we obtain
\begin{eqnarray*}
 \|w^{\ast}\|_{L^{\infty}(\Omega^*)} &\leq& \|w^{\ast}\|_{L^{\infty}(\p\Omega^*)} + C_n diam (\Omega^*) \left\|\frac{f(Du^{*})}{(\det {u^{*}} ^{ij})^{1/n}}\right\|_{L^n(\Omega^*)}\\
 &=& C + C \left(\int_{\Omega^{*}} \frac{|f|^n(Du^{*})}{ (\det D^2 u^*)^{-1}}~ dy\right)^{1/n}\\&=& 
 C + C  \left(\int_{\Omega} \frac{|f|^n(x)}{ \det D^2 u} \det D^2 u~ dx\right)^{1/n}  = C + C \|f\|_{L^n(\Omega)}.
\end{eqnarray*}
Since $w^{*} = G(d)- d G'(d)$ and the coercivity condition (B2) is satisfied, the above estimates give a uniform upper bound for $d=\det D^2 u$. The lower bound for $\det D^2 u$ follows from Lemma \ref{upwlem}.
\end{proof}

With Lemma \ref{boundd}, we can now complete the proof of the global $W^{4,p}$ estimates in Theorem \ref{keythm}.

\begin{proof}[Proof of theorem \ref{keythm}] The proof here is taken from \cite[Theorem 1.2]{Le}. We include it for completeness. 
It is an application of two regularity results: 
\begin{myindentpar}{1cm}
(i) Global H\"older continuity estimates for solutions of the linearized Monge-Amp\`ere equations \cite{Le} which are the global counterparts of
the fundamental interior H\"older estimates by
Caffarelli-Guti\'errez \cite{CG}, and \\
(ii) Global $C^{2,\alpha}$ estimates for the Monge-Amp\`ere equation \cite{TW08} when the Monge-Amp\`ere measure is only assumed to be globally $C^{\alpha}$.
\end{myindentpar}
By Lemma \ref{boundd}, 
$C^{-1}\leq \det D^2 u \leq C.$
Note that, by (\ref{AMCE}), $w$ is the solution to the linearized Monge-Amp\`ere equation $U^{ij}w_{ij} =f$ with boundary data $w=\psi.$ Because $\psi\in W^{2,p}(\Omega)$ with $p>n$, $\psi$ is clearly H\"older continuous on $\partial\Omega$. 
Thus, by \cite[Theorem 1.4]{Le}, $w \in C^{\alpha}(\overline{\Omega})$ for some $\alpha>0$ depending on the data of (\ref{AMCE})-(\ref{SBV}). Rewriting the equation for $w$ as
$$\det D^2 u = (G')^{-1}(w),$$
with the right hand side being in $C^{\alpha}(\overline{\Omega})$
and noticing $u=\varphi$ on $\p\Omega$ where $\varphi\in C^{3}(\overline{\Omega})$ because $\varphi\in W^{4,p}(\Omega)$ and $p>n$, we obtain
$u\in C^{2,\alpha}(\overline{\Omega})$ \cite[Theorem 1.1]{TW08}. Thus the first equation of (\ref{AMCE}) is a uniformly elliptic, second order partial differential equations in $w$
with $L^p$ right hand side. Hence $w\in W^{2,p}(\Omega)$ and in turn $u\in W^{4, p}(\Omega)$ with desired estimate
\begin{equation*}
\norm{u}_{W^{4,p}(\Omega)}\leq C,
\end{equation*}
where $C$ depends on $n, p, G, \Omega$, $\norm{f}_{L^{p}(\Omega)}, \norm{\varphi}_{W^{4,p}(\Omega)}, \norm{\psi}_{W^{2,p}(\Omega)}$, and $\inf_{\Omega} \psi.$
\end{proof}

\section{Proofs of the Remarks}
In this final section, for completeness, we give the proofs of Remarks \ref{ABrem}, \ref{smallpt}, \ref{nearb} and \ref{Lnsmall}.
\begin{proof}[Proof of Remark \ref{ABrem}] The proof of (b) is elementary and follows from direct computation so we skip it. For (a), suppose that $G$ satisfies (A1)-(A3). 
Let $$M(d)= G(d) - d G'(d)-\left[G(1)- w(1) + c(1-\frac{1}{n}) \log d\right].$$ Then, from $G''(d)= w'(d)$, and (A1), we have
$$
 M'(d) = -dw'-c(1-\frac{1}{n})\frac{1}{d}\geq (1-\frac{1}{n})w-c(1-\frac{1}{n})\frac{1}{d}.
$$
Thus, if $d\geq 1$ then by (A2), $M'(d)\geq 0.$ 
Therefore, for $d\geq 1$, we have $M(d)\geq M(1)=0$ and the result follows.
\end{proof}

\begin{proof}[Proof of Remark \ref{smallpt}] In all these examples, the function
$u$ is convex and of the form
$$u= v(r) = r^{\alpha},~r=|x|,~\text{with}~1<\alpha<2. $$
We can compute
$$\det D^2 u = v^{''}(\frac{v'}{r})^{n-1}=  \alpha^{n} (\alpha-1) r^{(\alpha-2)n},
w= (\det D^2 u)^{\theta-1}= [\alpha^n (\alpha-1)]^{\theta-1} r^{(\alpha-2)n (\theta-1)}\equiv W(r).
$$
From $\theta<1$ and $1<\alpha<2$, we have $w=0$ at $0\in\Omega$. Since $D^2 u$ and $(D^2 u)^{-1}$ are similar to 
$\text{diag}~ (v^{''}, \frac{v'}{r}, \cdots, \frac{v'}{r})$
and 
$\text{diag}~(\frac{1}{v^{''}}, \frac{r}{v^{'}}, \cdots, \frac{r}{v^{'}}),$ we can
compute
\begin{eqnarray*}U^{ij}w_{ij}= \frac{v^{''} (v^{'})^{n-1}}{r^{n-1}} \left(\frac{W^{''}}{v^{''}} + (n-1) \frac{W^{'}}{v^{'}}\right)
= \frac{[W^{'} (v^{'})^{n-1}]^{'}}{r^{n-1}}.
\end{eqnarray*}
It follows that for some $C_1= C_1(n,\alpha,\theta)>0$ and $C_2= C_2(n,\alpha,\theta)>0$,
\begin{align}\label{feq}
 W^{'} (v^{'})^{n-1} &= C_1r^{(\alpha-1)(n-1) + (\alpha-2)n (\theta-1)-1}= C_1r^{ (\alpha-2)(n\theta-1) + n-2},\nonumber\\
f&= U^{ij} w_{ij} =\frac{[W^{'} (v^{'})^{n-1}]^{'}}{r^{n-1}}= C_2r^{(\alpha-2)(n\theta-1) -2}.
\end{align}
(i) In this case, since $p/2<p<n$, we have 
$$\alpha = 1 + \frac{n-p}{2(3p-n)}\in (1,2)~\text{and}~\theta= \frac{1}{2n}(\frac{n}{p}-1) \in (0, \frac{1}{n}).$$
We find that $$(\alpha-2)(n\theta-1) -2= -\frac{1}{4} (1+3\frac{n}{p})>-n/p$$ and hence $f\in L^p(\Omega)$.
Since $\alpha-4<-2$, $|D^4 u|\sim r^{\alpha-4}\not\in L^{n/2}(\Omega)$  and hence $u\not\in W^{4, p}(\Omega)$.\\
(ii) Note that $u= r^{\alpha}$ where
$\alpha= 2 + \frac{2}{n\theta-1}\in (1,2).$
Then $f$ defined by (\ref{feq}) is a positive constant. However, as above, $u\not\in W^{4, n/2}(\Omega)$.\\
(iii) As in (ii), $u\not\in W^{4, n/2}(\Omega)$. In this case, we note that $f$ defined by (\ref{feq}) belongs to $L^p(\Omega)$ because
$(\alpha-2)(n\theta-1)-2= -1 -3n\theta/4>-1 -n\theta/2=-n/p.$
\end{proof}

\begin{proof}[Sketch of the proof of Remark \ref{nearb}] For $\theta<1/n$, (A1) and (A3) are satisfied, hence we have 
 the global gradient bound for $u$ as in Lemma \ref{Dubound}. The proof of Lemma \ref{Dubound} also shows that $\det D^2 u$ is uniformly bounded away from $0$
 and $\infty$ in a neighborhood $\Omega_{2\delta}$ of $\p\Omega$. Thus, as in \cite[Proposition 3.2]{S}, we find that $u$ separates quadratically from its tangent planes 
 on the boundary $\p\Omega$. Therefore, we can use 
\cite[Theorem 1.4]{Le} in $\Omega_{2\delta}$ to conclude that $w\in C^{\alpha} (\overline{\Omega_{3\delta/2}})$. By
the pointwise $C^{2,\alpha}$ estimates at the boundary for the 
Monge-Amp\`ere equation \cite{S}, applied to $\det D^2 u = (G')^{-1}(w)$,  we conclude that $u$ is $C^{2,\alpha}$ in $\overline{\Omega_{\delta}}$. Now, $U^{ij}\p_{ij}$ is a uniformly elliptic second order operator
with $C^{\alpha}$ coefficients in $\overline{\Omega_{\delta}}$. Thus, with $f\in L^p(\Omega)$, we have $w\in W^{2,p}(\Omega_{\delta})$. The $W^{4, p}$ estimates for $u$ in $\Omega_{\delta}$ follow.
\end{proof}

\begin{proof} [Proof of Remark \ref{Lnsmall}] As in the proof of Theorems \ref{mainthm} and \ref{keythm}, it suffices to prove a positive lower bound for $w$ when $\|f^{+}\|_{L^n(\Omega)}$ is small.
With $\theta<\frac{1}{n}$, (A3) is satisfied. Thus, by Lemma \ref{upwlem}, $w$  is uniformly bounded from above:
$\|w\|_{L^{\infty}(\Omega)} \leq C_1.$
Let $c_1:=\min_{\p\Omega}\psi>0$ Then, by recalling that
$\det U = (\det D^2 u)^{n-1} = w^{\frac{n-1}{\theta-1}},$
and using the ABP estimate to $U^{ij}w_{ij}=f\leq f^{+}$ in $\Omega$, we find
\begin{eqnarray*}\min_{\Omega} w&\geq& 
\min_{\partial\Omega} \psi -\frac{\text{diam}(\Omega)}{nw_n^{1/n}} \left\|\frac{f^{+}}{(\det U)^{1/n}}\right\|_{L^n(\Omega)}=
c_1 - C_2 \left\|f^{+}w^{\frac{n-1}{n(1-\theta)}}\right\|_{L^n(\Omega)} \\ &\geq& c_1 - C_2 \|f^{+}\|_{L^n(\Omega)} C_1 ^{\frac{n-1}{n(1-\theta)}}\geq c_1/2
\end{eqnarray*}
if
$\|f^{+}\|_{L^n(\Omega)} \leq c_1/\left(2C_2 C_1^{\frac{n-1}{n(1-\theta)}}\right).$
\end{proof}

\end{document}